\documentclass[a4paper,12pt]{article}
\usepackage{latexsym}
\usepackage{amsmath}
\usepackage{amssymb}
\usepackage{enumerate}

\usepackage{color}

\usepackage{theorem}
\newtheorem{theorem}{Theorem}[section]
\newtheorem{proposition}[theorem]{Proposition}
\newtheorem{lemma}[theorem]{Lemma}

\newtheorem{MT}{Main Theorem}

\theorembodyfont{\rmfamily}
\newtheorem{proof}{\textmd{\textit{Proof.}}}

\newtheorem{remark}[theorem]{Remark}

\makeatletter

\@addtoreset{equation}{section}
\makeatother

\newcommand{\qedd}{\hfill \Box}

\newcommand{\wt}{\widetilde}

\title{The Structure Theorem for The cut locus of 
a Certain Class of Cylinders of Revolution I
\footnote{
Mathematics Subject Classification (2010)\,: 53C22.}
.}
\author{Pakkinee CHITSAKUL}
\date{}
\begin{document}

\maketitle

\begin{abstract}
The aim of this paper is to determine the structure of the cut locus for a class of surfaces of revolution homeomorphic to a cylinder. Let $M$ denote a cylinder of revolution  which admits a reflective symmetry fixing a  parallel called the equator of $M.$ It will be proved that the cut locus of  a point  $p$ of $M$ is a subset of the union of the meridian and the parallel opposite to $p$ respectively,  if the Gaussian curvature of $M$ is decreasing on each upper half meridian.
\end{abstract}

\section{Introduction}
It is a very difficult problem to determine the structure of the cut locus 
of a Riemannian manifold and it was difficult  even for a quadric surface.

Since  Elerath (\cite{E}) succeeded in specifying the structure of the cut locus 
for paraboloids of revolution  and (2-sheeted) hyperboloids of revolution, 
the structures of the cut locus for quadric surfaces  of revolution have been 
studied.  After his work, Sinclair and Tanaka (\cite{ST})  determined the structure of the 
cut locus for   a class of surfaces of revolution containing the ellipsoids. Notice that
the structures of the cut locus for  triaxial ellipsoids  with unequal axes 
were also determined by Itoh and Kiyohara (\cite{IK}).

On the structure of the cut locus for a cylinder of revolution $(R^1\times S^1, dt^2+m(t)^2d\theta^2),$
Tsuji (\cite{Ts})  first determined the cut locus of a point on 
the equator $t=0$ if the cylinder 
 is symmetric with respect to
 the equator and the  Gaussian curvature is decreasing on the upper half 
 meridian $t>0,\theta=0.$
In 2003, Tamura (\cite{Ta}) determined the structure of the cut locus by 
adding an  assumption $m'\ne 0 $ except $t=0.$
In this paper, we determine the structure of the cut locus without this 
assumption. 

Here, let us review the notion  of a cut point and the cut locus of a point.
Let $\gamma :[0,a]\to M$ be a minimal geodesic segment in a complete Riemannian manifold $M.$ The end point of $\gamma(a)$ is called a {\it cut point} of $\gamma(0)$ along $\gamma,$
if any geodesic extension of $\gamma$ is not minimal anymore.
The {\it cut locus } $C_p$ of a point $p$ of $M$ is by definition the set of the cut points along all minimal geodesic segments emanating from $p.$ 

In this paper we will prove the following theorem.
\begin{MT}

Let $(M,ds^2)$ be a complete Riemannian manifold $R^1 \times S^1$ 
with a warped product metric $ds^2=dt^2+m(t)^2d\theta^2$ of the real line 
$(R^1,dt^2)$ and the unit circle $(S^1,d\theta^2)$.
Suppose that the warping function  $m$ is a positive-valued even 
function and the Gaussian curvature of $M$ is decreasing along 
the half meridian 
$t^{-1}[0,\infty) \cap \theta^{-1}(0)$.
If the Gaussian curvature of $M$ is positive on $t=0,$ then the structure of the cut 
locus $C_q$ of a point $q\in\theta^{-1}(0)$ in $M$ is given as follows:

\begin{enumerate}

\item
The cut locus $C_q$ is the union of a subarc of the  parallel $t=-t(q)$ opposite  to $q$ and the  meridian  opposite  to $q$ if  
$|t(q)|<t_0:= \sup \{t>0\: |\: m'(t)<0 \}$ and  $\varphi(m(t(q)))<\pi.$  More precisely,
$$C_q=\theta^{-1}(\pi) \cup \left( t^{-1}(-t(q)) \cap 
\theta^{-1}[\varphi(m(t(q))),2\pi-\varphi(m(t(q)))]\right).$$
\item
The cut locus $C_q$  is the  meridian $\theta^{-1}(\pi)$  opposite to  $q$ if   $\varphi(m(t(q)))\geq \pi$ or if  $|t(q)|\geq t_0.$

\end{enumerate}
Here,
the function $\varphi(\nu)$   on $(\inf m,m(0))$ is defined as
$$ \varphi(\nu):=2\int_{-\xi_(\nu)}^{0} 
\dfrac{\nu}{m \sqrt{m^2-\nu^2}}dt
=2\int_{0}^{\xi(\nu)} \dfrac{\nu}{m \sqrt{m^2-\nu^2}}dt,
$$
where $\xi(\nu):=\min\{t>0\: | \:  m(t)=\nu\}.$ 
Notice that  the point $q$ is an arbitrarily given point 
if the coordinates $(t,\theta)$ are chosen so as to satisfy $\theta(q)=0.$
\end{MT}

\begin{remark}
If the Gaussian curvature of  a cylinder of revolution is nonpositive everywhere, then
any geodesic has no conjugate point. Therefore,  it is clear to see that the 
 cut locus of a point on the manifold   is the 
 meridian opposite to the point. 

\end{remark}

\section{Preliminaries}
Let $f$ be the solution of the differential equation
\begin{equation}\label{eq:2.1}
f''+Kf=0
\end{equation}
with  initial conditions $f(0)=c$ and $f'(0)=0.$
Here $c$ denotes a fixed positive number and 
 $K:[0,\infty)\to R$ denotes  a  continuous function.

\begin{lemma}\label{lem2.1}
If $K(0)>0$ and $f'(t)\ne0$ for any $t>0$, then $f'(t)<0$ on $(0,\infty)$.
Furthermore, if $f>0$ on $[0,\infty)$, then $K(t)<0$ for some $t>0$.
\end{lemma}

\begin{proof}

Since $f''(0)=-K(0)f(0)<0$ by \eqref{eq:2.1},
$f '(t)$
 is strictly  decreasing on $(0,\delta)$ for some $\delta>0.$ This implies that 
 $0=f'(0)>f'(t) $ for any $t\in(0,\delta)$. Since $f'\ne0$ on 
 $[0,\infty),$  $f'(t)<0$ 
 on $(0,\infty)$. Furthermore, we assume that $f>0$ on $[0,\infty).$ 
Supposing that $K\geq 0$ on $ [0,\infty),$ we will get a  contradiction.
By \eqref{eq:2.1},
$$f''(t)=-K(t)f(t)\leq 0$$
on $[0,\infty).$ Hence $f'(t) $ is decreasing on $[0,\infty).$ In particular, 
$0=f'(0)>f '(\delta)\geq f '(t)$ for any $t\geq\delta.$ This contradicts the 
assumption $f>0.$
$\qedd$
\end{proof}

\begin{lemma}\label{lem2.2}
Suppose that $K(0)>0$ and $f>0$ on $[0,\infty)$. If $f'(t)=0$ for some $t>0$
 and $K$ is decreasing,
 then there exist a unique solution $t=t_0\in(0,\infty)$ of $f'(t)=0$ 
 such that $f'(t)<0$ on $(0,t_0)$ and $f'(t)>0$ on $(t_0,\infty)$ 
 and there exists  $t_1\in(0,t_0)$ satisfying $K(t_1)=0$. 
 Hence $K\geq0$ on $[0,t_1]$ and $K\leq0$ on $[t_1,\infty)$.
\end{lemma}

\begin{proof}
Let $a>0$ denote the minimum positive solution $t=a$ of $f'(t)=0.$
Suppose that there exist another solution $b(>a)$ satisfying $f'(b)=0$.
By the mean value theorem, there exist $t_1\in(0,a)$ and $s_1\in(a,b)$ 
satisfying $f''(t_1)=f''(s_1)=0$. Hence $K(t_1)=K(s_1)=0$ by \eqref{eq:2.1}.
Since $K$ is decreasing,  $K=0$ on $[t_1,s_1]$. 
Therefore, by \eqref{eq:2.1}, $f''(t)=0$ on $[t_1,s_1]$. 
In particular, $f'(a)=f'(t_1)=0.$ 
Since $0<t_1<a$, $t_1$ is a positive solution $t$ of $f'(t)=0$, 
which is less than $a$. This is a contradiction.
Therefore, there exists a unique positive solution $t=t_0$ of $f'(t)=0.$
 From the mean value theorem and \eqref{eq:2.1}, 
there exists $t_1\in(0,t_0)$ satisfying $K(t_1)=0.$  Since $K(t)$ is 
decreasing, $K\geq0$ on $[0,t_1]$ and $K\leq0$ on $[t_1,\infty).$ 
Hence by \eqref{eq:2.1},
$f''(t)=-K(t)f(t)\geq 0$ on $[t_1,\infty)$ and 
$f '(t)\geq f '(t_0)=0$ for any $t>t_0.$ 
Since $f'$ has a unique positive zero,  $f'>0$ on $(t_0,\infty).$
It is clear from the proof of Lemma \ref{lem2.1} that $f'<0$ on $(0,t_0).$
$\qedd$
\end{proof}
\section{Review of the behavior of geodesics}

From now on, $M$ denotes a complete Riemannian manifold $R^1\times S^1$ with 
a warped product Riemannian metric $ds^2=dt^2+m(t)^2d\theta^2$ of 
the real line $(R^1,dt^2)$ and the unit circle $(S^1,d\theta^2)$. 
Let us review the behavior of a geodesic
 $\gamma(s)=(t(s),\theta(s))$ on the manifold $M$.
For each unit speed geodesic $\gamma(s)=(t(s),\theta(s)),$
there exists a constant $\nu$ satisfying 
\begin{equation}\label {eq:3.1}
m(t(s))^2\theta'(s)=\nu.
\end{equation}
Hence, if $\eta(s)$ denotes the angle made by the velocity vector $\gamma'(s)$ of the  geodesic $\gamma(s)$ 
and the tangent vector $(\partial/\partial\theta)_{\gamma(s)},$ then 
\begin{equation}\label{cl}
m(t(s))\cos\eta(s)=\nu 
\end{equation}
 for any $s.$ The constant $\nu$ is called the 
{\it Clairaut constant} of $\gamma.$
The reader should refer to Chapter 7 in \cite{SST} for the Clairaut relation.
Since $\gamma(s)$ is unit speed, 
\begin{equation}\label{eq:3.2}
t'(s)^2+m(t(s))^2\theta'(s)^2=1
\end{equation}
holds. By \eqref{eq:3.1} and \eqref{eq:3.2},
it follows that 
\begin{equation}\label{eq:3.3}
t'(s)= \pm \dfrac{\sqrt{m(t(s))^2-\nu^2}}{m(t(s))}
\end{equation}
\begin{equation}\label{eq:3.4}
\theta(s_2)-\theta(s_1)= \varepsilon(t'(s)) \int_{t(s_1)}^{t(s_2)} 
\dfrac{\nu}{m \sqrt{m^2-\nu^2}}dt
\end{equation}
holds, if $t'(s)\neq 0$ on $(s_1,s_2)$ and $\varepsilon(t'(s))$ denotes 
the sign of $t'(s)$.

The length $L(\gamma)$ of a geodesic segment $\gamma(s)=(t(s),\theta(s))$, 
$s_1 \leq s \leq s_2$ is
\begin{equation}\label{eq:3.5}
L(\gamma)= \varepsilon(t'(s)) \int_{t(s_1)}^{t(s_2)} 
\dfrac{m(t)}{\sqrt{m(t)^2-\nu^2}}dt
\end{equation}
if $t'(s) \neq 0$ on $(s_1,s_2).$

From a direct computation, the Gaussian curvature $G$ of $M$ is given by
\begin{equation*}
G(q)=-\dfrac{m''}{m}(t(q))
\end{equation*}  
at each point $q\in M$. Since $G$ is constant on $t^{-1}(a)$ for 
each $a\in R$,
a smooth function $K$ on $R$ is defined by
\begin{equation*}
K(u):=G(q)
\end{equation*}  
for $q\in t^{-1}(u)$. Therefore $m$ satisfies the following 
differential equation 
\begin{equation*}
m''+Km=0
\end{equation*}  
with $m'(0)=0$.

From now on, we assume that the Gaussian curvature $G$ of $M$ 
is positive on $t^{-1}(0)$, and $m(t)=m(-t)$ holds for any $t\in R$.
Hence, $M$ is symmetric with respect to the equator $t=0$ and if $K$ 
is decreasing on $[0,\infty)$, then  by 
Lemma \ref{lem2.2},
$m'(t)<0$ for all $t>0$ or there exists a unique positive solution 
$t=t_0$ of $m'(t)=0$ such that
$m'<0 $ on $(0,t_0)$ and $m'>0$  on $(t_0,\infty).$ 
Furthermore, if the latter case happens, there exists $t_1\in(0,t_0)$ 
such that $K\geq 0$ on $[0,t_1]$ and $K\leq 0$ on $[t_1,\infty).$

For technical reasons, we treat both geodesics on $M$ and its 
universal covering space $\pi: \wt 
M\to M,$
where $\wt M:=(R^1\times R^1,d\tilde t^2+m(\tilde t)^2d\tilde\theta^2).$

Choose any point $p$ on the equator $t=0$. We may assume that $\theta(p)=0$ 
without loss of generality. Let $\gamma:[0,\infty) \to M$ denote
a geodesic emanating from $p=\gamma(0)$ with Clairaut constant 
$\nu \in (\inf m,m(0))$.
Notice that $\gamma$ is uniquely  determined up to the reflection with respect to $t=0.$
The geodesic $\gamma(s)=(t(s),\theta(s))$ is tangent to the parallel 
$t=\xi(\nu)$ ( if $t\circ\gamma)'(0)>0$ ) or $t=-\xi(\nu)$ ( if $(t\circ\gamma)'(0)<0$ ), where $\xi(\nu)>0$ denotes  the least 
positive  solution of 
$m(\xi(\nu))=\nu,$ that is,
\begin{equation*}
\xi(\nu):=\min \{ u>0\:  |\:  m(u)=\nu \}.
\end{equation*}

After $\gamma$ is tangent to the parallel $t=\xi(\nu)$ or $-\xi(\nu)$, 
$\gamma$ intersects the equator $t=0$ again.
Thus, after  $\tilde\gamma$  is tangent to the {\it parallel arc} $\tilde t=\xi(\nu)$ or 
$-\xi(\nu),$ $\tilde \gamma$ intersect $\tilde t=0$ again. Here
  $\tilde\gamma$ denotes a geodesic on $\wt M$ satisfying 
  $\gamma=\pi\circ\tilde\gamma.$

From \eqref{eq:3.4}, we obtain,
\begin{equation*}
\tilde\theta(s_0)-\tilde\theta(0)=  \int_{-\xi(\nu)}^{0} 
\dfrac{\nu}{m \sqrt{m^2-\nu^2}}dt
= \int_{0}^{\xi(\nu)} 
\dfrac{\nu}{m \sqrt{m^2-\nu^2}}dt,
\end{equation*}
and
\begin{equation*}
\tilde\theta(s_1)-\tilde\theta(s_0)=  \int_{-\xi(\nu)}^{0} 
\dfrac{\nu}{m \sqrt{m^2-\nu^2}}dt
= \int_{0}^{\xi(\nu)} 
\dfrac{\nu}{m \sqrt{m^2-\nu^2}}dt,
\end{equation*}
where  $s_0:=\min \{ s>0 \: |\: m(\tilde t(s))=\nu \} , 
        s_1:=\min \{ s>0 \: |\:  \tilde t(s)=0 \}$.

By summing up the argument above, we have,
\begin{lemma}\label{lem3.1}
Let $\tilde\gamma(s)=(\tilde t(s),\tilde\theta(s))$ denote a geodesic emanating 
from the point $\tilde p:=(\tilde  t,\tilde\theta)^{-1}(0,0)$ with Clairaut constant
 $\nu \in (\inf m,m(0))$. Then $\tilde\gamma$ intersects $\tilde t=0$
 again at  the point $(\tilde t,\tilde \theta)^{-1} (0,\varphi(\nu) ).$
Here,  
\begin{equation}\label{eq:3.6}
 \varphi(\nu):=2\int_{-\xi_(\nu)}^{0} 
\dfrac{\nu}{m \sqrt{m^2-\nu^2}}dt
=2\int_{0}^{\xi(\nu)} \dfrac{\nu}{m \sqrt{m^2-\nu^2}}dt.
 \end{equation}
\end{lemma}
\begin{lemma}\label{lem3.2}

The length $l(\nu)$ of the subarc $(\tilde t(s),\tilde\theta(s)),
 0\leq\tilde\theta(s)\leq\varphi(\nu),$ of $\tilde \gamma(s)$ is given by

\begin{equation}\label{eq:3.7}
l(\nu)=2\int ^0_{-\xi(\nu)} \frac{m}{\sqrt{ m^2-\nu^2   }} dt=2\int_{-\xi(\nu)}^{0} \frac{\sqrt{m^2-\nu^2}}{m}dt+\nu \varphi(\nu),
\end{equation}
and 
\begin{equation}\label{eq:3.8}
\frac{\partial  l}{\partial \nu} (\nu)=\nu\varphi'(\nu).
\end{equation}
\end{lemma}
\begin{proof}
From \eqref{eq:3.5}, we obtain, 
\begin{equation*}
l(\nu)=2\int ^0_{-\xi(\nu)} \frac{m}{\sqrt{ m^2-\nu^2   }} dt.
\end{equation*}
Since
$$\frac{m}{\sqrt{ m^2-\nu^2   }  }=\frac{\sqrt{ m^2-\nu^2   } }{m}+\frac{\nu^2}{m  
{\sqrt{ m^2-\nu^2  } }}$$
holds,
we get
\begin{equation*}
l(\nu)=2\int^0_{-\xi(\nu)} \frac{  {\sqrt{ m^2-\nu^2 }  }}{m}dt+2\int^0_{-\xi(\nu)}\frac{\nu^2}{m {\sqrt{ m^2-\nu^2  } }}dt.
\end{equation*}
Hence, by \eqref{eq:3.6}, we get \eqref{eq:3.7}.
By differentiating $l(\nu)$ with respect to $\nu,$ we get,
\begin{equation*}
l'(\nu)=2\int^0_{-\xi(\nu)} \frac{\partial}{\partial\nu}\frac{ \sqrt{ m^2-\nu^2   }}{m}dt+\varphi(\nu)+\nu\varphi'(\nu)=\nu\varphi'(\nu).
\end{equation*}

$\qedd$

\end{proof}


\section{The decline of the function $\varphi(\nu)$}

Let $\pi:\wt M=(R^1\times R^1, d\tilde t^2+m(\tilde t)^2d\tilde\theta^2)\to M$ denote the universal covering space of $M.$
 We choose an arbitrary point $\tilde p$ of $\tilde t^{-1}(-\infty,0],$  and we denote the cut locus of $\tilde p$ by $C_{\tilde p}.$
Before proving some lemmas on the cut locus, let us review the structure of the cut locus of $\wt M.$  We refer to \cite{ShT} or \cite{SST} on the structure of the cut locus of a 2-dimensional complete Riemannian manifold.

It is known that the cut locus has a local tree structure. Since $\wt M$ is simply connected, the cut locus has no circle.
If two cut points $x$ and $y$ are in a common connected component of the cut locus, then $x$ and $y$ are connected by a unique rectifiable arc in the cut locus.

Since $\wt M$ is homeomorphic to $R^2,$ we may define a global sector at each cut point. 
For general  surfaces,  only local sectors are defined (see \cite {ShT}, or \cite{SST}).  
A  global {\it sector} at each cut point $x$ of the point $\tilde p$  is by definition a connected component of $\wt M\setminus \Gamma_x ,$
where $\Gamma_x$ denotes the set  of all points lying on a minimal geodesic segment joining $\tilde p$ to $x.$ Let $c:[0,a]\to C_{\tilde p}$ denote a rectifiable arc in the cut locus.
Then for each cut point $c(t), t\in(0,a),$ $c$ bisects the sector at $c(t)$ containing $c[0,t)$ (respectively $c(t,a]$) .
For each sector of the point $\tilde p$ on $\wt M,$ there exists an end point of $C_{\tilde p},$ 
since $C_{\tilde p}$ has no circle. Here, a cut point $q$ of $\tilde p$ is called an {\it end point} if $q$ admits exactly one sector.

In this section, we assume 
that the Gaussian curvature $G$ of $M$ is increasing on the half meridian 
$t^{-1}(-\infty,0] \cap \theta^{-1}(0)$  and that $M$ has  a reflective symmetry with respect to $t=0.$
 Hence the Gaussian curvature of $\wt M$  is increasing on the lower  half meridian $\tilde t^{-1}(-\infty,0]\cap\tilde\theta^{-1}(0)$ and $\wt M$ has  a reflective  symmetry with respect to $\tilde t=0.$

\begin{lemma}\label{lem4.1}

Suppose that there exists a cut point  of the point  $\tilde p$ in 
$\tilde  t^{-1}(-\infty,0).$
Then  there exist two minimal geodesic segments $\alpha$ and $\beta$
joining $\tilde p$ to a cut point  $y$ of $\tilde p$  such that the global sector $D(\alpha,\beta)$ bounded by  $\alpha$ and $\beta$ has an end point of $C_{\tilde p}$ and 
 $D(\alpha,\beta)\subset\tilde t^{-1}(-\infty,0).$

\end{lemma}

\begin{proof}

Since the subset of cut points admitting at least two minimal geodesics 
is dense in the cut locus,  the existence of two minimal geodesics $\alpha$ and $\beta$ is clear
(see \cite{Bh}).
Since $\wt M$ has a reflective symmetry with respective to $\tilde t=0,$ it is trivial that 
 $D(\alpha,\beta)\subset \tilde t^{-1}(-\infty,0).$ 
Let $y$ denote the end point of $\alpha$ distinct from $\tilde p.$
Since the proof is complete in the case where the cut point $y$ is not an end point of the cut locus, we assume that $y$ is an end point. Then, we get an arc $c$ in the cut locus emanating from $y.$ Any interior point $y_1$ on $c$ is not an end point of the cut locus. It is clear that 
there exist two minimal geodesic segments joining $\tilde p$ and $y_1$ which bound a sector  containing $y$ as an end point of the cut locus.
$\qedd$
\end{proof}

\begin{lemma}\label{lem4.3}
For any unit speed minimal geodesic segment $\gamma:[0,L(\gamma)]\to \wt M$ 
joining $\tilde p $ to any end point $x$ of $C_{\tilde p}$ in 
the domain $D(\alpha,\beta),$ 
$x$ is conjugate to $\tilde p$ along $\gamma$ and $\gamma$ is shorter than 
$\alpha$ and $\beta.$ 
\end{lemma}

\begin{proof}

Note that for any end point $x$ of the cut locus, the set of all minimal 
geodesic segments joining $\tilde p$ to $x$ is connected. Therefore,
 $x$ is conjugate to $\tilde p$ along any minimal geodesic segments 
 joining $\tilde p$ to the end point of the cut locus.
Let $\gamma:[0,L(\gamma)]\to \wt M$ denote any minimal geodesic segment 
 $\tilde p$ to an end point $x$ of 
$C_{\tilde p}\cap D(\alpha,\beta).$ We will prove that $\gamma$ is shorter 
than $\alpha$ and $\beta.$
It follows from Theorem B in \cite{ShT} or \cite{IT} that there exists 
a unit speed arc 
$c:[0,l]\to C_{\tilde p}$ joining the end point $x$ to  $y,$
where $y$ denotes the end point  of $\alpha$
distinct from $\tilde p.$
Since the function $d(\tilde p, c(\tau))$ is a Lipschitz function, 
it follows from Lemma 7.29 in \cite{WZ} that
the function is differentiable for almost all $\tau$ and
\begin{equation}\label{eq:4-1}
d(\tilde p,c(l))-d(\tilde p,y)=\int_0^l \frac{d}{d\tau}
 d(\tilde p,c(\tau))d\tau
\end{equation}
holds.
From the Clairaut relation \eqref{cl}, the inner angle $\theta(\tau)$ at $c(\tau )$ of 
the sector containing  $c[0,\tau)$
 is less than $\pi.$ Hence, by the first variation formula, we get
$$\frac{d}{d\tau}d(\tilde p,c(\tau))=\cos\frac{\theta(\tau)}{2}>0$$
for almost all $\tau.$ 
Notice that for each each $\tau\in(0,l),$ the curve $c$ bisects the sector at $c(\tau)$ containing $c[0,\tau).$
Therefore, from \eqref{eq:4-1},
$$L(\alpha)=L(\beta)=d(\tilde p,c(l))>d(\tilde p,y)=L(\gamma).$$
$\qedd$
\end{proof}

\begin{lemma}\label{lem4.3}

Let $q$ be a point on $\tilde \theta^{-1}(0)$ and $u_0$ any real number. 
Then $d(q,c(\theta))$ is strictly increasing on $[0,\infty)$. 
Here $c:[0,\infty)\to \wt M$ denotes $c(\theta)=(u_0,\theta)$ in 
the coordinates $(\tilde t,\tilde \theta)$ and $d(\cdot,\cdot)$ denotes 
the Riemannian distance function on $\wt M.$

\end{lemma}

\begin{proof}
Choose any positive numbers $\theta_1<\theta_2.$ Let $\alpha_i, i=1,2,$
 denote minimal geodesic segments joining the point $q$ to $c(\theta_i)$ 
 respectively. Since $\theta_2>\theta_1,$  there exists  an intersection 
 $\alpha_2(t_2)$ of $\alpha_2$ and
the meridian $\tilde \theta=\theta_1.$ The point $c(\theta_1)$ is the 
unique nearest point on $\tilde t=u_0$ from $\alpha_2(t_2).$ Hence, 
$$d(\alpha_2(t_2),c(\theta_1))<d(\alpha_2(t_2),c(\theta_2)).$$
Therefore, by the triangle inequality,  we
get
$$d(q,c(\theta_2))=d(q,\alpha_2(t_2))+d(\alpha_2(t_2),c(\theta_2))>
d(q,\alpha_2(t_2))+d(\alpha_2(t_2),c(\theta_1))\geq d(q,c(\theta_1)).$$
This implies that $d(q,c(\theta))$ is strictly increasing on $[0,\infty).$
$\qedd$
\end{proof}

\begin{lemma}\label{lem4.4}
Suppose that  $\gamma:[0,L(\gamma)]\to \wt M$ is a minimal geodesic segment
 joining $\tilde p$ to an end point $x\in C_{\tilde p},$ which is a point 
 in the sector $D(\alpha,\beta)$ bounded by two minimal geodesic 
 segments $\alpha$ and $\beta$ emanating from $\tilde p.$
Then,  for any $s\in [0,L(\gamma)],$  
 $\tilde t(\alpha(s)) \geq\tilde t(\gamma(s)) \geq\tilde t(\beta(s))$ holds. 
Here we assume that 
\begin{equation*}
\angle 
( \alpha'(0) , ( {\partial}/{\partial \tilde t})_{\tilde p })
<\angle 
( \gamma'(0) , ( {\partial}/{\partial \tilde  t} )_{\tilde p })
<\angle 
( \beta'(0) , ( {\partial}/{\partial\tilde  t})_{\tilde p} ),
\end{equation*}
where  $\angle(\cdot,\cdot)$ denotes the angle made by two tangent
 vectors. 
\end{lemma}
\begin{proof}
From \eqref{eq:3.3}, it follows that for sufficiently small $s>0,$ 
$\tilde t(\alpha(s))>\tilde t(\gamma(s))>\tilde t(\beta(s))$ holds.
 Hence the set
$A:=\{s\in(0,L(\gamma))\: |\:  \tilde t(\alpha(s))>\tilde t(\gamma(s))>
\tilde t(\beta(s))\}$ 
 is a nonempty  open subset of $(0,L(\gamma)).$ 
Let $(0,s_0)$ denote the connected component of $A.$
It is sufficient to prove that $s_0=L(\gamma).$ Suppose that $s_0<L(\gamma).$
Thus, 
$\tilde t(\alpha(s_0))=\tilde t(\gamma(s_0))$ or
 $\tilde t(\gamma(s_0))=\tilde t(\beta(s_0))$ holds, since $A$ is open. 
By  applying Lemma \ref{lem4.3} for $u_0:=\tilde t(\alpha(s_0))$ and $\tilde t(\beta(s_0)),$ we get $\alpha(s_0)=\gamma(s_0)$ or $\gamma(s_0)=\beta(s_0),$
 which is a contradiction. 
$\qedd$
\end{proof}



\begin{lemma}\label{lem4.5} 

For any point  $\tilde p\in \tilde t^{-1}(-\infty,0],$ there does not exist a cut 
point of $\tilde p$
 in $\tilde t^{-1}(-\infty,0)$. 
In particular, the cut locus of $\tilde p$ is a subset of $\tilde t^{-1}(0)$ if $\tilde t(\tilde p)=0.$
This implies that the cut locus  $C_p$ of a point $p\in t^{-1}(0)$  is a subset of 
$\theta^{-1}(\pi)\cup t^{-1}(0)$. Here the coordinates $(t,\theta)$ are chosen so as to satisfy $\theta(p)=0.$

\end{lemma}
\begin{proof}
Suppose that  there exist a cut point of $\tilde p$ in 
$\tilde t^{-1}(-\infty, 0).$
By Lemma \ref{lem4.1}, 
there exist two minimal geodesic segments $\alpha$ and $\beta$ joining 
 a cut point 
$y$ of $\tilde p$ which bound a sector $D(\alpha,\beta)$ containing  
an end point $x$ of $C_{\tilde p} .$ Let 
 $\gamma :[0,L(\gamma)]\to \wt M$ be a unit speed geodesic segment joining 
 $\tilde p$ to the end point
$x.$
From Lemmas \ref{lem4.1} and  \ref{lem4.4},  it follows that for any $s\in [0,L(\gamma)]$,
\begin{equation*}
0\geq\tilde t(\alpha(s)) \geq \tilde t(\gamma(s)) \geq \tilde t(\beta(s))
\end{equation*}
holds. 
Since the Gaussian curvature $G$ is increasing on each lower half meridian,
we obtain
\begin{equation*}
G(\alpha(s)) \geq G(\gamma(s)) \geq G(\beta(s)).
\end{equation*}
By applying the Rauch comparison theorem for the pair of  geodesic segments $\alpha|_{[0,L(\gamma)]}$ and $\gamma,$ $\tilde p$ admits a conjugate point on 
$\alpha|_{[0,L(\gamma)]}$ along $\alpha$.

This contradicts the fact that $\alpha$ is minimal.  Since $\wt M$ is symmetric with respect to $\tilde t=0,$  the cut locus of $\tilde p$ is a subset of $\tilde t^{-1}(0),$ if
 $\tilde t(\tilde p)=0.$
This implies that $C_p\subset \theta^{-1}(\pi)\cup t^{-1}(0)$ for the point $p= t^{-1}(0)\cap\theta^{-1}(0).$
$\qedd$
\end{proof}

\begin{proposition}\label{prop4.6}
Let $M$ be a complete Riemannian manifold $R^1 \times S^1$ with 
a warped product metric $ds^2=dt^2+m(t)^2d\theta^2$ of the real line
 $(R^1,dt^2)$ and the unit circle $(S^1,d\theta^2)$. 
 Here the warping function $m:R \to (0,\infty)$ is a smooth even 
 function. If  the Gaussian curvature is positive on the equator and 
 decreasing on the  upper half meridian 
 $t^{-1}(0,\infty)\cap \theta^{-1}(0),$
then the function $\varphi(\nu)$ 
  is decreasing on $(\inf m,m(0))$.
 
\end{proposition}

\begin{proof}
Let $\wt M:=(R^1\times R^1, d\tilde t^2+m(\tilde t)^2d\tilde\theta^2)$ 
denote the universal covering space of $M.$
Choose any point $\tilde p$ on $\tilde t^{-1}(0).$
For each $\nu \in(\inf m,m(0)),$
let $\alpha_\nu:[0,\infty) \to \wt M$ denote the geodesic emanating 
from the point $\tilde p=\alpha_\nu (0)$  with Clairaut constant 
$\nu$ and with
$(\tilde t\circ\alpha_{\nu})'(0)<0.$
From the Clairaut relation, we get
$\angle (({\partial}/{\partial 
\tilde\theta})_{\tilde p} , \alpha'_\nu(0))=
\cos^{-1} {\nu}/{m(0)}$.
Choose any $\nu_1<\nu_2$ with $\nu_1,\nu_2\in(\inf m,m(0)).$
Since $$\cos^{-1}\frac{\nu_2}{m(0)}<\cos^{-1}\frac{\nu_1}{m(0)},$$ 
it follows from Lemma \ref{lem4.5} that $\alpha_{\nu_1}$ 
does not cross the domain bounded by 
the subarc of $\alpha_{\nu_2}$ and $\tilde t^{-1}(0)\cap\tilde\theta^{-1} 
[\tilde\theta(\tilde p),
\tilde\theta(\tilde p)+\varphi({\nu_2})]$.
 This implies that $\varphi(\nu_1)\geq \varphi(\nu_2).$
 Therefore, $\varphi(\nu)$ is decreasing on $(\inf m,m(0))$. 
$\qedd$
\end{proof}

\section{The cut locus of a point on $\wt M$}
 Choose any point $q$ on $\wt M$ with $-t_0<\tilde t(q)<0,$ where 
$t_0:=\sup\{\: t>0\: |  \:  m'(t)<0\}.$
Without loss of generality, we may assume 
 that $\tilde\theta(q)=0$. 
 We consider two geodesics $\alpha_\nu$ and $\beta_\nu$ emanating 
 from the point $q=\alpha_\nu(0)=\beta_\nu(0)$ with  Clairaut constant $\nu>0$. 
 Here we assume that
 \begin{equation*}
 \angle (({\partial}/{\partial \tilde t})_q,\alpha'_\nu (0)) >
 \angle (({\partial}/{\partial\tilde  t})_q,\beta'_\nu (0)).
 \end{equation*}

\begin{lemma}\label{lem5.1}
The two geodesics $\alpha_{\nu}$ and $\beta_\nu$ intersect again at the point $(\tilde t,\tilde \theta)^{-1}(u,\varphi(\nu) )$
if $\nu\in(\inf m,m(0)),$ where $u:=-\tilde t(q).$

\end{lemma}
\begin{proof}
Suppose that $\nu\in(\inf m,m(0)).$ 
Since $\alpha_\nu$ is tangent to the  parallel arc $\tilde t=-\xi(\nu)$, it follows 
from \eqref{eq:3.4} that
\begin{equation*}
 \tilde\theta(\alpha_\nu(s_1))-\tilde\theta(\alpha_\nu(0))=
 \int_{-\xi(\nu)}^{-u} \dfrac{\nu}{m\sqrt{m^2-\nu^2}}dt,
 \end{equation*}
where $s_1:=\min \{ s>0\: |\: \tilde t(\alpha_\nu(s))=-\xi(\nu) \} $,
and \begin{equation*}
\tilde \theta(\alpha_\nu(s_2))-\tilde\theta(\alpha_\nu(s_1))=
 \int_{-\xi(\nu)}^{u} \dfrac{\nu}{m\sqrt{m^2-\nu^2}}dt,
 \end{equation*}
 where $s_2:=\min \{  s>0\: |\: \tilde t(\alpha_\nu(s))=u\} $. 
Hence, we obtain,
\begin{equation}\label{eq:5.1}
\tilde \theta(\alpha_\nu(s_2))-\tilde\theta(\alpha_\nu(0))=
 \int_{-\xi(\nu)}^{u} \dfrac{\nu}{m\sqrt{m^2-\nu^2}}dt +
 \int_{-\xi(\nu)}^{-u} \dfrac{\nu}{m\sqrt{m^2-\nu^2}}dt.
 \end{equation}
Since $m$ is an even function,
\begin{equation*}
\int^u_{-\xi(\nu)} \frac{\nu}{m\sqrt{m^2-\nu^2}} dt=\int^0_{-\xi(\nu)}\frac{\nu}{m\sqrt{m^2-\nu^2}} dt+
\int^0_{-u}\frac{\nu}{m\sqrt{m^2-\nu^2}} dt
\end{equation*}
holds.
Therefore, by \eqref{eq:5.1}, 
$$\tilde\theta(\alpha_\nu(s_2))-\tilde\theta(\alpha_\nu(0)) =2\int^0_{-\xi(\nu)}\frac{\nu}{ \sqrt{m^2-\nu^2  }   }dt=\varphi(\nu).$$
This implies that $\alpha_\nu$ passes through the point $(\tilde t,\tilde\theta)^{-1}(u,\varphi(\nu)).$
On the other hand, after  $\beta_\nu$ is tangent to $\tilde t=\xi(\nu)$ at $\beta_\nu(s_1^+),$ where $s_1^+:=\min\{s>0 \: | \: \tilde t(\beta_\nu(s)) =\xi(\nu)\}$,
the geodesic intersects $\tilde t=u$ again at $\beta_\nu(s_2^+),$
where $s_2^+:=\min\{s>s_1^+\: | \: \tilde t(\beta_\nu(s))=u\}.$
By the similar computation as above, we get
\begin{equation*}
\tilde\theta(\beta_\nu(s_2^+))-\tilde\theta(\beta_\nu(0))=\varphi(\nu). 
\end{equation*}
This implies that  $\alpha_\nu$ and $\beta_\nu$ pass through the common  point $(\tilde t,\tilde\theta)^{-1}(u,\varphi(\nu)).$
$\qedd$
\end{proof}

\begin{lemma}\label{lem5.2}
The two  geodesic segments $\alpha_\nu |_{[0,s_2]}$ and $\beta_{\nu}|_{[0,s_2^+]}$ have the same length and its length equals $l(\nu), $ which is defined in Lemma \ref{lem3.2}.
In particular, $s_2=s_2^+.$
Here, $s_2$ and $s_2^+$ denote the numbers defined in the proof of Lemma \ref{lem5.1}.

\end{lemma}
\begin{proof}
From \eqref{eq:3.5}, we have
\begin{equation}\label{eq:5.2}
 L( \alpha_\nu |_{[0,s_1]})=
 \int_{-\xi(\nu)}^{-u} \dfrac{m}{\sqrt{m^2-\nu^2}}dt,
 \end{equation}
and
\begin{equation*}
 L( \alpha_\nu |_{[s_1,s_2]} )=
 \int_{-\xi(\nu)}^{u} \dfrac{m}{\sqrt{m^2-\nu^2}}dt=
 \int_{-\xi(\nu)}^{0} \dfrac{m}{\sqrt{m^2-\nu^2}}dt +
 \int_{0}^{u} \dfrac{m}{\sqrt{m^2-\nu^2}}dt,
 \end{equation*}
where  $s_1$ denotes the number defined in the proof of Lemma \ref{lem5.1}.
Since $m$ is even
\begin{equation}\label{eq:5.3}
L(\alpha_\nu|_{[s_1,s_2]} )=\int_{-\xi(\nu)}^{0} \frac{m}{\sqrt{m^2-\nu^2}}dt+
 \int_{-u}^{0} \frac{m}{\sqrt{m^2-\nu^2}}dt.
 \end{equation}
Therefore, we get, by \eqref{eq:3.7}, \eqref{eq:5.2} and \eqref{eq:5.3},
\begin{equation*}
L( \alpha_\nu |_{[0,s_2]}   )=
2\int_{-\xi(\nu)}^{0} \dfrac{m}{\sqrt{m^2-\nu^2}}dt=l(\nu).
\end{equation*}
Analogously we have,
\begin{equation*}
L(\beta_\nu|_{[0,s_2^+]})= l(\nu).
\end{equation*}

$\qedd$
\end{proof}

\begin{lemma}\label{lem5.3}
Let $q$ be  a point on $\wt M$ with $|\tilde t(q)|\in(0,t_0).$ 
Then, for any $\nu\in(\inf m,m(u)],$ where $u=-\tilde t(q),$
$\alpha_\nu|_{[0,s_2(\nu)]}$ and $\beta_\nu|_{  [0,s_2(\nu)]}$ are minimal geodesic segments joining $q$ to the point $(\tilde t,\tilde \theta)^{-1}(u,\tilde \theta(q)+\varphi(\nu)), $ and 
in particular,
$\{(\tilde t,\tilde \theta) \: | \: \tilde t=u, \tilde \theta\geq \varphi(m(u))+\tilde\theta(q)\} $
is a subset of the cut locus of the point $q.$
Here,   $s_2(\nu):=\min\{s>0 \: |\:  \tilde t(\alpha_\nu(s))=u\}$ for each $\nu\in(\inf m,m(0)).$ 
\end{lemma}
\begin{proof}
Without loss of generality, we may assume that $\tilde\theta(q)=0.$
We will prove that $\alpha_{\nu}|_{[0,s_2(\nu)]}$ is a minimal geodesic segment joining $q$ to the point $\alpha_{\nu}(s_2(\nu))=(\tilde t,\tilde\theta)^{-1}(u,\varphi(\nu)).$
Suppose that $\alpha_{\nu_0}|_{[0,s_2(\nu_0)]}$ is not minimal for some $\nu_0 \in (\inf m,m(u)].$
Here we assume that $\nu_0$ is the minimum solution $\nu=\nu_0$ of 
$\varphi(\nu)=\varphi(\nu_0).$

Let  $\alpha:[0,d(q,x)]\to M$ be a minimal geodesic  segment  joining $q$ to $x:=\alpha_{\nu_0}(s_2(\nu_0))=(\tilde t,\tilde\theta)^{-1}(u,\varphi(\nu_0)).$
Hence,  $\varphi(\nu_1)=\varphi(\nu_0)=\tilde\theta(x)$ and 
$\alpha $ equals $\alpha_{\nu_1}|_{[0,s_2(\nu_1)]}$ or $\beta_{\nu_1}|_{[0,s_2(\nu_1)]},$
 where
$\nu_1 \in(\inf m,m(0))$ 
denotes the Clairaut constant of $\alpha.$
By Proposition \ref{prop4.6}, $\varphi(\nu)=\varphi(\nu_0)$ for any
 $\nu\in[\nu_0,\nu_1].$
 Hence, by Lemmas \ref{lem3.2} and \ref{lem5.2}
 we get,
\begin{equation*}
s_2(\nu_1)=L(\alpha)=L(\alpha_{\nu_1}|_{[0,s_2(\nu_1)]})=L(\alpha_{\nu_0}|_{[0,s_2(\nu_0)]})=s_2(\nu_0).
\end{equation*}
This implies that $\alpha_{\nu_0}|_{[0,s_2(\nu_0)]}$ is minimal, which is a  contradiction, since we assumed that  $\alpha_{\nu_0}|_{[0,s_2(\nu_0)]}$ is not minimal.
Therefore, by Lemma \ref{lem5.2},   
for any $\nu\in(\inf m,m(u)],$
the  geodesic segments $\alpha_{\nu}|_{[0,s_2(\nu)]}$ and 
$\beta_{\nu}|_{[0,s_2(\nu)]}$
 are minimal geodesic segments joining $q$ to the point $(\tilde t,\tilde\theta)^{-1}(u,\varphi(\nu))=\alpha_\nu(s_2(\nu)).$
In particular,   the point 
$\alpha_\nu(s_2(\nu))=\beta_\nu(s_2(\nu))   $ 
is a cut point of $q.$
$\qedd$
\end{proof}

\begin{proposition}\label{prop5.4}
The cut locus of the point $q$ in Lemma \ref{lem5.3} equals the set
$$ 
\{(\tilde t,\tilde \theta) \: | \: \tilde t=u, \tilde \theta\geq |\varphi(m(u))|\}. $$
Here the coordinates $(\tilde t , \tilde\theta)$ are chosen so as to satisfy $\tilde\theta(q)=0.$
\end{proposition}
\begin{proof}
By Lemma \ref{lem5.3},
geodesic segments $\alpha_\nu|_{[0,s_2(\nu)]}$ and $\beta_\nu|_{[0,s_2(\nu)]}$
are minimal geodesic segments for any $\nu\in(\inf m,m(u)].$
Hence their limit geodesics $\alpha^{-}:=\alpha_{\inf m}$ and $\beta^+:=\beta_{\inf m}$ are 
rays, that is, any  their  subarcs are minimal.

Since $\wt M$ has a reflective symmetry with respect to $\tilde\theta=0,$
it is trivial from Lemma \ref{lem5.3} that 
the set 
$\{(\tilde t,\tilde \theta) \: | \: \tilde t=u, \tilde \theta\geq |\varphi(m(u))|\}$
is  a subset of the cut locus of $q.$
Suppose that there exists a cut point $y\notin 
\{(\tilde t,\tilde \theta) \: | \: \tilde t=u, \tilde \theta\geq| \varphi(m(u))| \}.$
Without loss of generality, we may assume that $\tilde\theta(y)>0=\tilde\theta(q)$ and $\tilde  t(q)=-u<0.$
From Lemma \ref{lem4.5},
$\tilde t(y)>0$  and $y$  is not a point in the unbounded domain cut off by two rays $\alpha^-$ and $\beta^+,$ and hence the point lies in the domain $D^+$ cut off by  $\beta^+$ and the submeridian $\tilde t>-u, \tilde \theta= \tilde \theta(q)=0.$
Since the cut locus of $C_q$ has  a tree structure,
there exists an end point $x$ of the cut locus in the $D^+.$
Hence, $x$ is conjugate to $q$
for any minimal geodesic segment $\gamma$ joining $q$ to $x.$ Since such a minimal geodesic $\gamma$ runs in the domain $D^+,$ the Clairaut constant of the segment is positive and  less than $\inf m.$
From  the Clairaut relation \eqref{cl}, any geodesic cannot be tangent to any parallel arc $\tilde t=c,$ if the Clairaut constant is positive and less than $\inf m.$
From Corollary  7.2.1 in \cite{SST}, $\gamma$ has no conjugate point of $q,$ which is a contradiction.
$\qedd$
\end{proof}

\begin{lemma}\label{lem5.5}
Let $q$ be  a point on $\wt M$ with $|\tilde t(q)|\geq t_0.$
Then the cut locus of $q$ is empty. 
\end{lemma}
\begin{proof}
Suppose that the cut locus of  a point $q$ with $|\tilde t(q)|\geq t_0$ is nonempty.
Since $\wt M$ has a reflective symmetry with respect to $\tilde t=0,$ we may assume that $\tilde t(q)\leq -t_0.$ Hence by Lemma \ref{lem4.5}, there exists an end point $x$ of the cut locus  $C_q$ in $\tilde t^{-1}(0,\infty).$
Let $\gamma : [0,d(q,x)]\to \wt M$ denote a minimal geodesic segment joining $q$ to $x.$ Then $x$ is conjugate to $q$ along $\gamma,$ since $x$ is  an end point of $C_q.$
Since $\tilde \theta(x)>0=\tilde \theta (q),$ the Clairaut constant $\nu$ of $\gamma$ is positive, by \eqref{eq:3.1}. Moreover,
from the Clairaut relation \eqref{cl}, the Clairaut constant $\nu$  is  less than $\inf m=m(t_0),$ since $\gamma$ intersects $\tilde t=-t_0.$ Therefore, $\gamma$ cannot be tangent to any parallel arc $\tilde t=c.$ From Corollary  7.2.1 in \cite{SST}, $\gamma$ has no conjugate point of $q,$ which is a contradiction.
$\qedd$
\end{proof}
Now our Main theorem is clear from Proposition \ref{prop5.4} and Lemma \ref{lem5.5}.
\bigskip

{\noindent\large\bf Acknowledgments}

I would like to express my gratitude to Professor Minoru TANAKA 
who kindly gave me guidance for the lectures and numerous comments.


\bigskip

\begin{center}
Pakkinee CHITSAKUL\\

\bigskip

Department of Mathematics\\
King Mongkut's Institute of Technology Ladkrabang\\
Ladkrabang, Bangkok\\
10\,--\,520 Thailand

\medskip
{\tt kcpakkin@kmitl.ac.th}
\end{center}

\end{document}